\documentclass[a4paper,12pt]{amsart}

\usepackage{amssymb}
\usepackage{amsmath}
\usepackage{amsthm}
\usepackage{graphics}
\usepackage{hyperref}

\DeclareFontFamily{U}{rsfs}{\skewchar\font"7F}
\DeclareFontShape{U}{rsfs}{m}{n}{
	<-6> rsfs5
	<6-8> rsfs7
	<8-> rsfs10
	}{}
\DeclareMathAlphabet{\mathscr}{U}{rsfs}{m}{n}

\theoremstyle{plain}
\newtheorem{theorem}{Theorem}[section]
\newtheorem{prop}[theorem]{Proposition}
\newtheorem{lemma}[theorem]{Lemma}

\theoremstyle{definition}

\theoremstyle{remark}
\newtheorem{ass}{Assumption}

\newtheorem{remark}{Remark}[section]

\def\real{{\mathbb R}}

\def\integer{{\mathbb Z}}

\def\supp{\mathop{\mathrm{supp}}\nolimits}

\numberwithin{equation}{section}

\newcount\hour \newcount\minute
\hour=\time \divide \hour by 60
\minute=\time

\def\now{%
\ifnum \hour<13
  \ifnum \hour=0 \advance \hour by 12 \number\hour:\else \number\hour:\fi%
     \ifnum \minute<10 0\fi%
     \number\minute%
\ A.M.%
\else \advance \hour by -12 \number\hour:%
  \ifnum \minute<10 0\fi%
  \number\minute%
  \ P.M.%
\fi%
}

\title{Dynamic entropic repulsion for interacting interfaces}
\author{Takao~Nishikawa}
\address{Department of Mathematics, College of Science and Technology, Nihon University, 1-8-14 Kanda-Surugadai, Chiyoda-ku, Tokyo 101-8308, Japan}
\email{nisikawa@math.cst.nihon-u.ac.jp}
\keywords{entropic repulsion, Ginzburg-Landau model, effective interfaces,
massless fields}
\subjclass{60K35, 82C24, 35K55}

\begin{document}

\begin{abstract}
	The dynamic entropic repulsion for the Ginzburg-Landau $\nabla\phi$
	interface model was discussed in \cite{DN03}
	and the asymptotics of the height of the interface was identified.
	This paper studies a similar problem for two interfaces on the wall
	which are interacting with one another by the exclusion rule.
	Each leading order of the asymptotics of height
	is $\sqrt{\log t}$ as $t\to\infty$ for the system on $\integer^d,\,d\ge3$,
	$\log t$ for the system on $\integer^2$.
	The coefficient of the leading term for each interface is also identified.
\end{abstract}

\maketitle

\section{Introduction}
Under the coexistence of distinct phases, interfaces
separating these phases are formed. One of problems on the interface separating
phases is the study of the effect of a hard wall. Once imposing the
effect of the hard wall on random interfaces, interfaces are pushed
up by the fluctuation. The problem
to identify how high the interface is pushed up is called entropic
repulsion. 

The paper \cite{DN03} investigated the such problem in the dynamical situation
and identified the asymptotic behavior of the height for
the Ginzburg-Landau $\nabla\phi$ interface model on the hard wall.
This dynamics can be regarded as the motion of the interface
separating two phases and reflected by the hard wall. 
In this paper, let us discuss the entropic repulsion for
two interfaces separating three distinct phases, which are reflected by
the hard wall and interacting with one another by the exclusion.
We shall consider the stochastic interface $\Phi_t=\{(\phi_t^{(1)}(x),
\phi_t^{(2)}(x));\,x\in\integer^d\}$
governed by the following system of SDEs of Skorokhod type:
\begin{equation}\label{SDE1}
\left\{\begin{split}
	&d\phi_t^{(1)}(x)=\Delta\phi_t^{(1)}(x)dt+\sqrt{2}dw_t^{(1)}(x)
	+d\ell_t^{(1)}(x)-d\ell_t^{(2)}(x), \\
	&d\phi_t^{(2)}(x)=\Delta\phi_t^{(2)}(x)dt+\sqrt{2}dw_t^{(2)}(x)
	+d\ell_t^{(2)}(x), \\
	&\text{$\ell_t^{(1)}(x)$ and $\ell_t^{(2)}(x)$ are increasing in $t$,}
	\qquad\qquad\qquad\qquad\qquad\qquad\qquad\qquad  \\
	&0\le \phi_t^{(1)}(x)\le \phi_t^{(2)}(x),\quad t\ge0, \\
	&\int_0^\infty \phi_t^{(1)}(x)\,d\ell_t^{(1)}(x)=0, \\
	&\int_0^\infty (\phi_t^{(2)}(x)-\phi_t^{(1)}(x))\,d\ell_t^{(2)}(x)=0, \\
\end{split}\right.
\end{equation}
for $x\in\integer^d$, where the operator $\Delta$ is the discrete
Laplacian on $\integer^d$, that is,
\[\Delta\zeta(x)=\sum_{
\begin{subarray}{c}
y\in\integer^d, \\
|x-y|=1
\end{subarray}}
(\zeta(y)-\zeta(x))\]
for $\zeta=\{\zeta(x)\in\real; x\in\integer^d\}\in\real^{\integer^d}$.
Here, $\{w_t^{(i)}(x);\,x\in\integer^d,i=1,2\}$ is the family
of independent one-dimensional standard Brownian motions.
The processes $\ell_t^{(1)}(x)$ and $\ell_t^{(2)}(x)$ are called the
local time for $\phi_t^{(1)}(x)$ and $\phi_t^{(2)}(x)-\phi_t^{(1)}(x)$,
respectively.

We assume the following condition on the initial data
throughout this paper.
\begin{ass}
We assume that the initial data of $\Phi_t$ is distributed
by a independent identical distribution $\prod_{x\in\integer^d}
\rho(d\phi^{(1)}(x),d\phi^{(2)}(x))$.
We moreover assume that $\rho$ satisfies the following condition:
\begin{enumerate}
\item the support of $\rho$ is in $\{
(u,v)\in\real^2;\, 0\le u\le v\}$,
\item $\rho$ has a finite second moment,
\item $\rho$ is absolutely continuous with respect to the Lebesgue measure
on $\real^2$.
\end{enumerate}
\end{ass}
The conditions (i) and (ii) is quite natural
for the existence of the solution of \eqref{SDE1}.
The condition (iii) is rather technical requirement. It is imposed
in order to calculate the relative entropy, see Section~\ref{sec-lower}.
We now state our main result which gives
the precise asymptotic behavior of height of interfaces.
\begin{theorem}\label{main-thm}
We have for $i=1,2$
\begin{gather}
	\lim_{t\to\infty}P\left(\phi_t^{(i)}(0)\in
	\left(\sqrt{(C_i/2-\epsilon)\log_d(t)},
	\sqrt{(C_i/2+\epsilon)\log_d(t)}\right)\right)=1 \label{eq1.4}
\end{gather}
with two constant $C_1,C_2>0$, where $\log_d(t)$ is defined by
\begin{equation*}
	\log_d(t)=\begin{cases}
		(\log t)^2,& d=2, \\
		\log t,& d\ge3.
	\end{cases}
\end{equation*}
Here, the constants $C_1,C_2$ are explicitly identified as
\begin{align*}
	C_1&=\begin{cases}
		\displaystyle\int_0^\infty p_t(0,0)\,dt,& d\ge 3, \\
		\displaystyle\lim_{t\to\infty}\frac{1}{\log t}\int_0^t p_s(0,0)\,ds,& d=2,
	\end{cases} \\
	C_2&=(\sqrt{2}+1)^2C_1,
\end{align*}
where $\{p_t(x,y)\}$ is the transition kernel of the simple random walk
on $\integer^d$ generated by $2\Delta$.
\end{theorem}

This result says that the lower interface grows up like
$\sqrt{C_1\log_d(t)/2}$ as $t\to\infty$ and the upper
grows up like $(\sqrt{2}+1)\sqrt{C_1\log_d(t)/2}$. We note that
the gap between the lower and upper grows up like $\sqrt{C_1\log_d(t)}$
and it is different
from asymptotics of the single interface, see \cite{DN03}.
From this fact, one can see that the upper interface is pushed up more
by the exclusion from another interface than the exclusion from the hard
wall.

We need to remark that this result is related to the entropic repulsion
of the finite volume
Gibbs measures on $\Lambda_N:=[-N,N]^d\cap\integer^d$ of multi-layered system.
It is introduced by $\mu_{\Lambda_N}^{+}(\cdot)=\mu_{\Lambda_N}(\cdot|0\le\phi^{(1)}\le\phi^{(2)})$, where $\mu_{\Lambda_N}$ is given by
\begin{align*}
\mu_{\Lambda_N}(d\phi^{(1)}d\phi^{(2)})=
\frac{1}{Z_{\Lambda_N}}&\exp\left(\sum_{i=1,2}\sum_{x\in\Lambda}
\phi^{(i)}(x)\Delta\phi^{(i)}(x)\right) \\
&\times\prod_{x\in\Lambda_N}d\phi^{(1)}(x)d\phi^{(2)}(x)
\prod_{x\in\integer^d\smallsetminus\Lambda_N}\delta_0(d\phi^{(1)}(x))\delta_0(d\phi^{(2)}(x))
\end{align*}
where $Z_{\Lambda_N}$ is the normalizing constant. 
In \cite{BG03} (when $d\ge3$) and \cite{Sa03} (when $d=2$),
the asymptotic behavior of sample mean is identified as follows:
\begin{align*}
	\lim_{N\to\infty}&\mu_{\Lambda_N}\left(|\Lambda_{\kappa N}|^{-1}
	\sum_{x\in \Lambda_{\kappa N}}
	\phi^{(i)}(x)\in
	\left(\sqrt{(C_i-\epsilon)\log_d(N)},
	\sqrt{(C_i+\epsilon)\log_d(N)}\right)\right)=1
\end{align*}
holds for every $i=1,2,\,\eta>0$ and $0<\kappa<1$.
Here, the constants $C_1,C_2$ is same as in Theorem~\ref{main-thm}.
We can easily to see that the asymptotic behavior stated above 
coincides with \eqref{eq1.4} with taking $t=N^{2\pm\epsilon}$.

\begin{remark}
	We require the asymptotic behavior stated above to show
	Theorem~\ref{main-thm}.
	Since such asymptotic behavior of the height of the interface
	distributed by the Gibbs measure is unknown in general,
	we only think of the Gaussian system \eqref{SDE1}.
\end{remark}

We shall prove Theorem~\ref{main-thm} by a similar method to \cite{DN03}.
We split the proof of Theorem~\ref{main-thm} into two parts,
the lower bound and the upper bound of the height of interfaces,
see Theorems~\ref{lower-bound} and \ref{upper-bound}, respectively.
For the lower bound, we derive it by reducing our problem to that for
the finite volume Gibbs measure with applying the comparison theorem
and the logarithmic Sobolev inequality. 
Though we have the pointwise estimate for the finite volume Gibbs measure
in the case of the single interface (see \cite{BDZ95}), we have only
the estimate for the sample mean, see \cite{BG03} and \cite{Sa03}
for details, and it is not enough for our goal.
However, since we can show the lower bound for the expectation
value with the help of the result of \cite{BG03} and \cite{Sa03},
we shall establish the estimate for the variance and show that
the order of the fluctuation is relatively small to that of the expectation
value, see Proposition~\ref{asymp-var} for details.
The proof of the upper bound is rather simple. We construct the suitable dynamics which always stays above \eqref{SDE1} and give the upper bound
for the introduced with applying the result of \cite{DN03}.

The organization of this paper is the following. In Section~\ref{sec-dyn},
we study several properties of \eqref{SDE1} and the dynamics on a finite set
corresponding to \eqref{SDE1}. In Section~\ref{sec-lower} and \ref{sec-upper},
we shall give the lower bound and the upper bound of the height
of the interface, respectively.

\section{Dynamics on an infinite set and on a finite set}\label{sec-dyn}
\subsection{Notations}
Before starting the discussion, we shall prepare several notations
which are used later.

Let $(\integer^d)^*$ be the set of all directed bonds $b=(x,y),\,
x,y\in\integer^d,|x-y|=1$ in $\integer^d$.
We write $x_b=x$ and $y_b=y$ for $b=(x,y)$.
For every subset $\Lambda$ of $\integer^d$, we denote
the set of all directed bonds included $\Lambda$ and touching $\Lambda$
by $\Lambda^*$ and $\overline{\Lambda^*}$, respectively.
That is,
\begin{align*}
	\Lambda^*&:=\{b\in(\integer^d)^*;\,x_b\in\Lambda
	\text{ and }y_b\in\Lambda\},\\
	\overline{\Lambda^*}&:=\{b\in(\integer^d)^*;\,x_b\in\Lambda
	\text{ or }y_b\in\Lambda\}.
\end{align*}

For a height variable $\phi=\{\phi(x);\,x\in\integer^d\}\in\real^{\integer^d}$,
we define the (discrete) gradient operator $\nabla$ by
$\nabla\phi(b):=\phi(x)-\phi(y)$ for $\phi\in\real^{\integer^d}$
and $b=(x,y)\in(\integer^d)^*$.
Here, we note that the discrete Laplacian has the following expression:
\[\nabla\phi(x)=-\sum_{b\in(\integer^d)^*:x_b=x}\nabla\phi(b),
\quad \phi\in\real^{\integer^d}.\]

We denote $\phi\le\psi$ for $\phi,\psi\in\real^{\integer^d}$,
if $\phi(x)\le\psi(x)$ holds for every $x\in\integer^d$.
Similarly, for $\Phi=(\phi^{(1)},\phi^{(2)}),\Psi=(\psi^{(1)},\psi^{(2)})\in
\real^{\integer^d}\times\real^{\integer^d}$,
we denote $\Phi\le\Psi$
if $\phi^{(1)}\le\psi^{(1)}$ and $\phi^{(2)}\le\psi^{(2)}$ hold.
We denote the configuration $\phi\in\real^{\integer^d}$ such that
$\phi(x)=0$ for every $x\in\integer^d$ simply by $0$.

Let us define a space of height $\mathscr{X}^2_r\,(r>0)$ by
\begin{align*}
	\mathscr{X}^2_r&=\left\{\Phi=(\phi^{(1)},\phi^{(2)})\in
	\real^{\integer^d}\times\real^{\integer^d};\,\|\phi\|_{r}^2:=
	\sum_{i=1,2}\sum_{x\in\integer^d}|\phi^{(i)}(x)|^2\exp(-2r|x|)<\infty
	\right\}.
\end{align*}
We note that $\mathscr{X}^2_r$ is Hilbert space with the inner product
\[(\Phi,\Psi)_r=
\sum_{i=1,2}\sum_{x\in\integer^d}\phi^{(i)}(x)\psi^{(i)}(x)\exp(-2r|x|)\]
for $\Phi=(\phi^{(1)},\phi^{(2)}),\Psi=(\psi^{(1)},\psi^{(2)})
\in\real^{\integer^d}\times\real^{\integer^d}$.
We define the space of height variables $\mathscr{X}^2_{r,+}\, (r>0)$ by 
$\mathscr{X}^2_{r,+}=\{\Phi=(\phi^{(1)},\phi^{(2)})\in\mathscr{X}^2_r;
0\le\phi^{(1)}\le\phi^{(2)}\}$, which is the actual state space of the solution $\Phi_t$
of \eqref{SDE1}.

For $\Phi=(\phi^{(1)},\phi^{(2)})\in\real^{\integer^d}\times\real^{\integer^d}$, we say ``$\Phi$ has a compact support'' when there exists a
finite $\Lambda\subset\integer^d$ such that
\[\phi^{(i)}(x)=0,\quad x\in\integer^d\smallsetminus\Lambda,i=1,2\]
and we then denote the smallest $\Lambda$ by $\supp\Phi$.

\subsection{The dynamics on a finite set}
We introduce the dynamics $\Phi_t^{\Lambda}
=(\phi_t^{(1),\Lambda},\phi_t^{(2),\Lambda})$ on finite $\Lambda$
by the following system of SDEs of Skorokhod type:
\begin{equation}\label{SDE2}
\begin{cases}
	d\phi_t^{(1),\Lambda}(x)=\Delta\phi_t^{(1),\Lambda}(x)dt+\sqrt{2}dw_t^{(1)}(x)
	+d\ell_t^{(1),\Lambda}(x)-d\ell_t^{(2),\Lambda}(x),& x\in\Lambda, \\
	d\phi_t^{(2),\Lambda}(x)=\Delta\phi_t^{(2),\Lambda}(x)dt+\sqrt{2}dw_t^{(2)}(x)
	+d\ell_t^{(2),\Lambda}(x), &x\in\Lambda, \\
	\phi_t^{(1),\Lambda}(x)=\phi_0^{(1),\Lambda}(x),&x\in\integer^d\smallsetminus\Lambda \\
	\phi_t^{(2),\Lambda}(x)=\phi_0^{(2),\Lambda}(x),&x\in\integer^d\smallsetminus\Lambda \\
	0\le \phi_t^{(1)}(x)\le \phi_t^{(2)}(x),& t\ge0,x\in\integer^d \\
	\text{$\ell_t^{(1)}(x)$ and $\ell_t^{(2)}(x)$ are increasing in $t$,} &\\
	\int_0^\infty \phi_t^{(1)}(x)\,d\ell_t^{(1)}(x)=0, & \\
	\int_0^\infty (\phi_t^{(2)}(x)-\phi_t^{(1)}(x))\,d\ell_t^{(2)}(x)=0. &
\end{cases}
\end{equation}
There exists the unique strong solution $\Phi^{\Lambda}_t
=(\phi^{(1),\Lambda}_t,\phi^{(2),\Lambda}_t)$
of \eqref{SDE1} with an arbitrary initial data satisfying
$0\le\phi^{(1)}_0\le\phi^{(2)}_0$.
When $\Lambda=\Lambda_N:=[-N,N]^d\cap\integer^d$, we simply denote $\Phi_t^{\Lambda_N}(x),\,\phi_t^{(i),\Lambda_N}(x)$ and $\ell_t^{(i),\Lambda_N}(x)$ by $\Phi_t^{N}(x),\,\phi_t^{(i),N}(x)$ and $\ell_t^{(i),N}(x)$, respectively.

We should note that the comparison theorem holds for the dynamics \eqref{SDE1}.
It will help us to show the lower bound.
\begin{prop}\label{comparison3}
Let $\Phi_t=(\phi^{(1)}_t,\phi^{(2)}_t)$ and $\Phi_t^\Lambda=
(\phi^{(1),\Lambda}_t,\phi^{(2),\Lambda}_t)$ be solutions of \eqref{SDE1}
and \eqref{SDE2} with common Brownian motions $\{w_t^{(i)}(x);\,x\in\integer^d,i=1,2\}$. We assume that $\Phi^{\Lambda}_0$ and $\Phi_0$ satisfy
$\phi^{(i),\Lambda}_0(x)=\phi^{(i)}_0(x)$ for $x\in\Lambda$
and $i=1,2$ and $\phi^{(i),\Lambda}_0(x)=0$ for $x\in\integer^d\smallsetminus
\Lambda$. Then, $\Phi_t^\Lambda\le\Phi_t$ holds for every $t\ge0$.
\end{prop}
\begin{proof}
	Since $0=\phi^{(i),\Lambda}_t(x)\le\phi^{(i)}_t(x)$ holds
	outside of $\Lambda$, it is sufficient 
	to show $\phi^{(i),\Lambda}_t(x)\le\phi^{(i)}_t(x)$ for $x\in\Lambda$.
	We define $\psi^{(i)}(x)=\phi^{(i),\Lambda}_t(x)-\phi^{(i)}_t(x)$
	for $x\in\integer^d$ and $i=1,2$. We then have 
	{\allowdisplaybreaks\begin{align*}
		&\sum_{i=1,2}\sum_{x\in\Lambda}
		\left(\psi^{(i)}_t(x)^{+}\right)^2 \\
		&\quad \le 2\sum_{i=1,2}\sum_{x\in\Lambda}
		\int_0^t \psi^{(i)}_s(x)^{+}
		\Delta\psi^{(i)}_s(x)\,ds \\
		&\qquad {}+
		2\sum_{x\in\Lambda}\int_0^t
		\psi^{(i)}_s(x)^{+}\left(d\ell^{(1),\Lambda}_t(x)-
		d\ell^{(1)}_t(x)\right)
		\\
		&\qquad {}+2\sum_{x\in\Lambda}\int_0^t
		\left(\psi^{(2)}_s(x)^{+}-\psi^{(1)}_s(x)^{+}\right)
		\left(d\ell^{(2),\Lambda}_t(x)-
		d\ell^{(2)}_t(x)\right) \\
		&\quad =:F_{1}+F_{2}+F_{3}.
	\end{align*}}%
	Here, by the definition of the local time $\ell^{(1),\Lambda}_t(x)$
	and $\ell^{(1)}_t(x)$, we obtain the second term $F_2$ is nonpositive.
	Noting that
	\[\left\{(a-b)^{+}-(c-d)^{+}\right\}\left\{(a-c)^{+}-(b-d)^{+}\right\}\ge0\]	holds for every $a,b,c,d\in\real$, the sign of $\psi^{(2)}_s(x)^{+}-\psi^{(1)}_s(x)^{+}$ is same as of
	\[\left(\phi_s^{(2),\Lambda}(x)-\phi_s^{(1),\Lambda}(x)\right)^+
	-\left(\phi_s^{(2)}(x)-\phi_s^{(1)}(x)\right)^+.\]
	We therefore obtain that the third term $F_3$ is also nonpositive.
	Using the summation-by-parts formula
	\begin{equation}\label{summation-by-parts}
		\sum_{x\in\Lambda}\phi(x)\Delta\psi(x)
		=-\sum_{b\in\overline{\Lambda^*}}\nabla\phi(b)\nabla\psi(b)
	\end{equation}
	for every $\phi,\psi\in\real^{\integer^d}$ such that $\phi(x)=0$
	on $\integer^d\smallsetminus\Lambda$ and
	\[(u^+-v^+)(u-v)\ge(u^+-v^+)^2\]
	for every $u,v\in\real$, we obtain that the first term $F_1$ is
	nonpositive.
	Summarizing above, we obtain $\psi_t^{(i)}(x)=0$ for every $t\ge0$
	and $x\in\Lambda$, which shows the conclusion.
\end{proof}

\subsection{Finite volume Gibbs measures and Dirichlet forms}
We introduce Gibbs measures associated with \eqref{SDE2}.
However, there does not exist an infinite volume Gibbs measure.
We  therefore introduce finite volume Gibbs measures only.

We at first introduce the finite volume Gibbs measure for the case
without any interaction. We define a probability
measure $\nu_{\Lambda,\xi}$ on $\real^\Lambda$ by
\[\nu_{\Lambda,\xi}(d\phi)=Z_{\Lambda,\xi}^{-1}
\exp(-H(\phi\wedge\xi))\,\prod_{x\in\Lambda}d\phi^+(x),\]
where $d\phi(x)$ is the Lebesgue measure on $[0,\infty)$.
This measure $\nu$ is the finite volume Gibbs measure corresponding
to the single interface on the wall.
We define $\bar\mu_{\Lambda,\xi}$ by $\bar\mu_{\Lambda,\xi}=
\nu_{\Lambda,\xi_1}\times\nu_{\Lambda,\xi_2}$, which is the finite volume
Gibbs measure associated with two interfaces on the wall
without any interaction. 
We define $\mu_{\Lambda,\xi}$ by $\mu_{\Lambda,\xi}=
\bar\mu_{\Lambda,\xi}(\cdot|\phi^{(2)}\ge\phi^{(1)})$, which is the finite
volume Gibbs measure corresponding to the dynamics \eqref{SDE2}
with $\xi\equiv\phi_0^\Lambda$. 
Since we only deal with the case $\xi\equiv0$, we simply denote
$\mu_{\Lambda,0}$ by $\mu_{\Lambda}$ if no confusion arises.

We shall introduce an approximating sequence for $\mu_{\Lambda}$
and the equation \eqref{SDE2}.
Let $\rho\in C^\infty_0(\real)$ be non-negative,
symmetric ($\rho(u)=\rho(-u)$), $\int_\real \rho(u)\,du=1$
and $\rho(u)=0$ if $|u|\ge1$. For $0<\delta\le1$, we define $\rho_\delta$
by $\rho_\delta(u):=\delta^{-1}\rho(\delta^{-1}u)$ for $u\in\real$.
We define $\chi_{\alpha}$ and $\chi_{\alpha,\delta}$ by
$\chi_{\alpha}(u)=\alpha^{-1}(u^-)^2/2$ and 
$\chi_{\alpha,\delta}(u)=(\rho_\delta\ast\chi_{\alpha}(\cdot+\delta))(u)$,
respectively.
For $\epsilon=(\epsilon_1,\epsilon_2),\,\epsilon_1,\epsilon_2>0$
and $\delta>0$, we define the probability measure $\mu_{\Lambda,\epsilon}$
and $\mu_{\Lambda,\epsilon,\delta}$ on $(\real^\Lambda)^2$ by
\[\frac{d\mu_{\Lambda,\epsilon}}{d\bar\mu_{\Lambda}}=
	Z^{-1}_{\Lambda,\epsilon}\exp\left(-\sum_{x\in\Lambda}W_{\epsilon}(\phi^{(1)}(x),\phi^{(2)}(x))\right),\]
and
\[\frac{d\mu_{\Lambda,\epsilon,\delta}}{d\bar\mu_{\Lambda}}=
Z^{-1}_{\Lambda,\epsilon,\delta}\exp\left(-\sum_{x\in\Lambda}W_{\epsilon,\delta}(\phi^{(1)}(x),\phi^{(2)}(x))\right),\]
respectively, where 
$W_{\epsilon}$ and $W_{\epsilon,\delta}$ is defined by
\begin{align*}
W_{\epsilon}(u,v)&=\chi_{\epsilon_1}(u)+\chi_{\epsilon_2}(v-u), \\
W_{\epsilon,\delta}(u,v)&=\chi_{\epsilon_1,\delta}(u)
+\chi_{\epsilon_2,\delta}(v-u)
\end{align*}
for $u,v\in\real$.
Here, $Z_{\Lambda,\epsilon}$ and $Z_{\Lambda,\epsilon,\delta}$
are normalizing constants.
We introduce the Dirichlet form $\mathscr{E}_\epsilon^\Lambda$ and 
$\mathscr{E}_{\epsilon,\delta}^\Lambda$ by
\[\mathscr{E}_\epsilon^\Lambda(F,G)=
\sum_{i=1,2}\sum_{x\in\Lambda}
\int_{(\real^\Lambda)^2}
\frac{\partial F}{\partial \phi^{(i)}(x)}
\frac{\partial G}{\partial \phi^{(i)}(x)}d\mu_\epsilon(d\phi)\]
and
\[\mathscr{E}_{\epsilon,\delta}^\Lambda(F,G)=
\sum_{i=1,2}\sum_{x\in\Lambda}
\int_{(\real^\Lambda)^2}
\frac{\partial F}{\partial \phi^{(i)}(x)}
\frac{\partial G}{\partial \phi^{(i)}(x)}d\mu_{\epsilon,\delta}(d\phi)\]
for $F,G\in C^2((\real^\Lambda)^2)$, respectively.
We note that $\mathscr{E}_\epsilon^\Lambda$ is the Dirichlet form
associated to SDEs \begin{equation}\label{eq2.1}
\left\{\begin{split}
	d\phi_t^{(1),\Lambda,\epsilon}(x)&=\Delta
	\phi_t^{(1),\Lambda,\epsilon}(x)\,dt+\sqrt{2}dw_t^{(1)}(x)
	+d\ell_t^{(1),\Lambda,\epsilon}(x)-d\ell_t^{(2),\Lambda,\epsilon}(x), \\
	d\phi_t^{(2),\Lambda,\epsilon}(x)&=
	\Delta\phi_t^{(2),\Lambda,\epsilon}(x)\,dt+\sqrt{2}dw_t^{(2)}(x)
	+d\ell_t^{(2),\Lambda,\epsilon}(x)
\end{split}\right.
\end{equation}
for $x\in\Lambda$, where $\ell_t^{(1),\epsilon}(x)$ $\ell_t^{(2),\epsilon}(x)$ are defined by
\begin{align*}
	\ell_t^{(1),\Lambda,\epsilon}(x)&=
	\epsilon_1^{-1}\int_0^t
	\phi_s^{(1),\Lambda,\epsilon}(x)^{-}\,ds, \label{eq4.2} \\
	\ell_t^{(2),\Lambda,\epsilon}(x)&=
	\epsilon_2^{-1}\int_0^t
	\left(\phi_s^{(2),\Lambda,\epsilon}(x)-\phi_s^{(1),\Lambda,\epsilon}(x)\right)^{-}\,ds. 
\end{align*}
We also note that $\mathscr{E}^\Lambda_{\epsilon,\delta}$ is the 
the Dirichlet form associated to SDEs
\begin{equation}\label{eq2.2}
\left\{\begin{split}
	d\phi_t^{(1),\Lambda,\epsilon,\delta}(x)&=\Delta
	\phi_t^{(1),\Lambda,\epsilon,\delta}(x)\,dt+\sqrt{2}dw_t^{(1)}(x)
	+d\ell_t^{(1),\Lambda,\epsilon,\delta}(x)-d\ell_t^{(2),\Lambda,\epsilon,\delta}(x), \\
	d\phi_t^{(2),\Lambda,\epsilon,\delta}(x)&=
	\Delta\phi_t^{(2),\Lambda,\epsilon,\delta}(x)\,dt+\sqrt{2}dw_t^{(2)}(x)
	+d\ell_t^{(2),\Lambda,\epsilon,\delta}(x)
\end{split}\right.
\end{equation}
for $x\in\Lambda$, where $\ell_t^{(1),\epsilon}(x)$ $\ell_t^{(2),\epsilon}(x)$ are defined by
\begin{align*}
	\ell_t^{(1),\Lambda,\epsilon}(x)&=
	-\int_0^t
	\chi'_{\epsilon_1,\delta}
	\left(\phi_s^{(1),\Lambda,\epsilon,\delta}(x)\right)\,ds, \label{eq4.2} \\
	\ell_t^{(2),\Lambda,\epsilon}(x)&=
	-\int_0^t\chi'_{\epsilon_2,\delta}
	\left(\phi_s^{(2),\Lambda,\epsilon,\delta}(x)
	-\phi_s^{(1),\Lambda,\epsilon,\delta}(x)\right)\,ds. 
\end{align*}
It is easy to show that there exist unique strong solutions
for \eqref{eq2.1} and \eqref{eq2.2}, respectively,
since all coefficients are Lipschitz continuous.
Noting that $W_{\epsilon,\delta}$ converges
to $W_{\epsilon}$ uniformly in $\real^2$ for fixed
$\epsilon=(\epsilon_1,\epsilon_2),\,\epsilon_1,\epsilon_2>0$,
it is easy to see the following identity:
\begin{equation}\label{eq2.4}
	\lim_{\delta\downarrow0}
	E\left[\sum_{x\in\Lambda}\sum_{i=1,2}
	\left|\phi_t^{(i),\Lambda,\epsilon,\delta}(x)
	-\phi_t^{(i),\Lambda,\epsilon}(x)\right|^2\right]=0
\end{equation}
for every $t\ge0$. 
We can also show that
$\Phi_t^{\Lambda,\epsilon}$ converges to $\Phi_t^\Lambda$
as $\epsilon_1\downarrow0$ and $\epsilon_2\downarrow0$
in the following sense:
\begin{equation}\label{eq2.3}
		\lim_{\epsilon_2\downarrow0}
		\lim_{\epsilon_1\downarrow0}
		E\left[\sum_{x\in\Lambda}\sum_{i=1,2}
		\left|\phi_t^{(i),\Lambda}(x)
		-\phi_t^{(i),\Lambda,\epsilon}(x)\right|^2\right]=0
\end{equation}
for every $t>0$ and $\Lambda\subset\integer^d$, see Theorem~\ref{existence-finite} for details.
Furthermore, we can show
\begin{equation}\label{eq2.3d}
		\lim_{\Lambda\uparrow\integer^d}
		E\left[\sum_{i=1,2}
		\left|\phi_t^{(i)}(x)
		-\phi_t^{(i),\Lambda}(x)\right|^2\right]=0
\end{equation}
for every $x\in\integer^d$ and $t>0$, which will be claimed
in Theorem~\ref{existence-infinite}.
These identities imply that the estimate for the variance of $\phi_t$
can be reduced to that for $\phi_t^{\Lambda,\epsilon,\delta}$.
We shall discuss an important property
of $\phi_t^{\Lambda,\epsilon,\delta}$,
comparison theorem with respect to initial datum.
This plays the key role in the proof of Theorem~\ref{asymp-var}.
\begin{prop}\label{semigr-lip}
	Let $\Phi_t^{\Lambda,\epsilon,\delta}$ and 
	$\tilde\Phi_t^{\Lambda,\epsilon,\delta}$ be the solution
	of \eqref{eq2.2} with initial data satisfying
	$\Phi_0^{\Lambda,\epsilon,\delta}\le
	\tilde\Phi_0^{\Lambda,\epsilon,\delta}$. We then have
	$\Phi_t^{\Lambda,\epsilon,\delta}\le
	\tilde\Phi_t^{\Lambda,\epsilon,\delta}$ and
	\begin{equation}\label{2.11b}
		\sum_{i=1,2}\left(
		\tilde\phi_t^{(i)\Lambda,\epsilon,\delta}(x)
		-\phi_t^{(i),\Lambda,\epsilon,\delta}(x)\right)
		\le \sum_{y\in\Lambda}\sum_{i=1,2}p_t(x,y)\left(
		\tilde\phi_0^{(i)\Lambda,\epsilon,\delta}(y)
		-\phi_0^{(i),\Lambda,\epsilon,\delta}(y)\right)
	\end{equation}
	for every $t>0$ and $x\in\Lambda$.
\end{prop}
\begin{proof}
	We shall at first show the first inequality in \eqref{2.11b}.
	We define $\psi^{(i)}$ by
	$\psi_t^{(i)}=\tilde\phi_t^{(i),\Lambda,\epsilon,\delta}
	-\phi_t^{(i),\Lambda,\epsilon,\delta}$.
	Differentiating
	$\sum_{i=1,2}\sum_{x\in\Lambda}
	(\psi^{(i)}_t(x)^{-})^2$
	in $t$ and integrating in $t$, we obtain
	{\allowdisplaybreaks\begin{align*}
		&\sum_{i=1,2}\sum_{x\in\Lambda}
		\left(\psi^{(i)}_t(x)^{-}\right)^2 \\
		&\quad \le -2\sum_{i=1,2}\sum_{x\in\Lambda}
		\int_0^t \psi^{(i)}_s(x)^{-}
		\Delta\psi^{(i)}_s(x)\,ds \\
		&\qquad {}-
		2\sum_{x\in\Lambda}\int_0^t
		\psi^{(1)}_s(x)^{-}
		\left(\chi'_{\epsilon_1,\delta}
		\left(\phi_s^{(1),\Lambda,\epsilon,\delta}(x)\right)
		-\chi'_{\epsilon_1,\delta}
		\left(\tilde\phi_s^{(1),\Lambda,\epsilon,\delta}(x)\right)\right)\,ds
		\\
		&\qquad {}-2\sum_{x\in\Lambda}\int_0^t
		\left(\psi^{(2)}_s(x)^{-}-\psi^{(1)}_s(x)^{-}\right) \\
		&\qquad\qquad\qquad 
		\times\left(\chi'_{\epsilon_2,\delta}
		\left(\phi_s^{(2),\Lambda,\epsilon,\delta}(x)
		-\phi_s^{(1),\Lambda,\epsilon,\delta}(x)\right)
		-\chi'_{\epsilon_2,\delta}
		\left(\tilde\phi_s^{(2),\Lambda,\epsilon,\delta}(x)
		-\tilde\phi_s^{(1),\Lambda,\epsilon,\delta}(x)\right)\right)\,ds \\
		&=:F_{1}+F_{2}+F_{3}.
	\end{align*}}%
	Noting that $\chi'_{\epsilon,\delta}(u)$ is nondecreasing in $u\in\real$
	and that
	\[\left\{(a-b)^{-}-(c-d)^{-}\right\}\left\{(a-c)^{-}-(b-d)^{-}\right\}\ge0\]	holds for every $a,b,c,d\in\real$, we obtain that
	$F_2$ and $F_3$ in the right hand side are nonpositive.
	In order to obtain the first inequality in \eqref{2.11b},
	we only need to show
	that the first term $F_1$ is nonpositive.
	Using \eqref{summation-by-parts} and
	$(u-v)(u^--v^-)\le -(u^--v^-)^2$ for every $u,v\in\real^d$, we get
	{\allowdisplaybreaks\begin{align*}
		\sum_{x\in\Lambda}&
		\psi^{(i)}_s(x)^{-}
		\Delta\psi^{(i)}_s(x)\ge \frac{1}{2}\sum_{b\in\overline{\Lambda^*}}
		\left(\nabla\left(\psi^{(i)}\right)^{-}(b)\right)^2\ge 0.
	\end{align*}}%
	We have used $\psi^{(i)}_s(x)^{-}=0$ for every $x\in\integer^d
	\smallsetminus\Lambda$.
	This shows that the term $F_1$ is nonpositive.

	We next show the inequality \eqref{2.11b}.
	Assuming $\alpha_s(x,y)$ be continuous differentiable in
	$s\in[0,t]$, we then obtain
	\begin{align*}
		\sum_{i=1,2}\alpha_t(y)\psi_t^{(i)}(y)
		&=\sum_{i=1,2}\alpha_0(y)\psi_0^{(i)}(y) \\
		&\quad {}+\sum_{i=1,2}\int_0^t
		\frac{d}{ds}\alpha_s(y)\psi_s^{(i)}(y)\,ds \\
		&\quad {}+\sum_{i=1,2}\int_0^t \alpha_s(y)
		\Delta\psi_s^{(i)}(y)\,ds \\
		&\quad {}-\int_0^t \alpha_s(y)\left(
		\chi'_{\epsilon_1,\delta}(\tilde\phi_s^{(1),\Lambda,\epsilon}(y))
		-\chi'_{\epsilon_1,\delta}(\phi_s^{(1),\Lambda,\epsilon}(y))\right)\,ds
	\end{align*}
	Note that
	\begin{align*}
		\frac{\partial W_{\epsilon,\delta}}{\partial u}+
		\frac{\partial W_{\epsilon,\delta}}{\partial v}=
		\chi'_{\epsilon_1,\delta}(u)
	\end{align*}
	holds by definition of $W_{\epsilon,\delta}$.
	Since $\chi'_{\epsilon_1,\delta}(u)$ is nondecreasing in $u\in\real$,
	we obtain that the fifth terms is nonpositive.
	Taking the sum over $y\in\Lambda$, we get
	\begin{align*}
		\sum_{i=1,2}\sum_{x\in\Lambda}\alpha_t(y)\psi_t^{(i)}(y) 
		&\le\sum_{i=1,2}\sum_{x\in\Lambda}\alpha_0(y)\psi_0^{(i)}(y) \\
		&\qquad {}+\sum_{i=1,2}\sum_{x\in\Lambda}\int_0^t
		\frac{d}{ds}\alpha_s(y)\psi_s^{(i)}(y)\,ds \\
		&\qquad {}+\sum_{i=1,2}\sum_{x\in\Lambda}\int_0^t \Delta\alpha_s(x)
		\psi_s^{(i)}(x)\,ds.
	\end{align*}
	Here, we have applied \eqref{summation-by-parts} twice.
	Choosing $\alpha_s(y)=p_{t-s}^{\Lambda}(x,y)$,
	we obtain the second and third terms cancel out,
	since $\alpha_s(y)$ is the solution of the forward equation
	\[
	\begin{cases}
		\frac{d}{dt}\alpha_s(y)=-\Delta\alpha_s(y),&y\in\Lambda, \\
		\alpha_s(y)=0,&y\in\integer^d\smallsetminus\Lambda,\,t\ge0 \\
		\alpha_t(y)=\delta_x(y),& y\in\Lambda.
	\end{cases}
	\]
	Using trivial inequality $p_t^\Lambda(x,y)\le p_t(x,y)$, we get
	the conclusion.
\end{proof}

\subsection{Approximation scheme for the dynamics on the finite set}
\label{sec-approx}
For a finite subset $\Lambda\subset\integer^d$, let us consider the dynamics 
$\Phi_t^\Lambda=(\phi_t^{(1),\Lambda},\phi_t^{(2),\Lambda})$
governed by SDEs \eqref{SDE2}.
In this subsection, we omit the domain $\Lambda$ for simplicity of notations
when no confusion arises.

Since the state space of \eqref{SDE2} is $|\Lambda|$-fold of two dimensional
cones and therefore convex, we can apply Tanaka's result \cite{T79},
and immediately obtain the existence and uniqueness of the strong solution for
our equation. However, in order to show the lower bound,
we need the nice approximation to the solution, see Section~\ref{sec-lower}.
We shall introduce an approximation scheme for solutions for our equation
\eqref{SDE2} from that for the equations with weak reflection terms.
Our goal in this subsection is the following:
\begin{theorem}\label{existence-finite}
	Let $\Phi_t=(\phi_t^{(1)},\phi_t^{(2)})$ and
	$\Phi^\epsilon_t=(\phi_t^{(1),\epsilon},\phi_t^{(1),\epsilon})$
	be the solutions for \eqref{SDE2} and \eqref{eq2.1}
	with common Brownian motions and
	initial datum $\Phi_0=(\phi^{1}_0,\phi^{2}_0)$, respectively.
	We then have
	\begin{equation}\label{eq5.1b}
	\phi_t^{(i)}(x)=\lim_{\epsilon_2\downarrow0}\lim_{\epsilon_1\downarrow0}\phi_t^{(i),\epsilon}(x)
	\end{equation}
	holds almost surely for every $i=1,2,\,t>0$ and $x\in\Lambda$.
	Furthermore, we obtain the identity \eqref{eq2.3}.
\end{theorem}

Our first observation is on the monotonicity for $\Phi_t^{\epsilon}$
in $\epsilon_1$, which plays key role in the proof
of Theorem~\ref{existence-finite}.
\begin{lemma}\label{lem5.2a}
	Fix $\epsilon=(\epsilon_1,\epsilon_2),
	\epsilon'=(\epsilon'_1,\epsilon'_2)$ such that
	$\epsilon_1\ge\epsilon'_1$ and $\epsilon_2=\epsilon'_2$.
	Let $\Phi_t^{\Lambda,\epsilon}$ and $\Phi_t^{\Lambda,\epsilon'}$
	be solutions of \eqref{eq2.1} with common Brownian motions and
	initial datum.
	We then have $\Phi_t^{\Lambda,\epsilon}\le\Phi_t^{\Lambda,\epsilon'}$
	for every $t>0$.
\end{lemma}
\begin{proof}
	The proof is quite parallel to that of Lemma~\ref{semigr-lip}.
	We define $\psi_t^{(i)}$ by
	$\psi_t^{(i)}=\phi_t^{(i),\epsilon}
	-\phi_t^{(i),\epsilon'}$.
	Differentiating
	$\sum_{i=1,2}\sum_{x\in\Lambda}
	(\psi^{(i)}_t(x)^{-})^2$
	in $t$ and integrating in $t$, we obtain
	{\allowdisplaybreaks\begin{align*}
		&\sum_{i=1,2}\sum_{x\in\Lambda}
		\left(\psi^{(i)}_t(x)^{-}\right)^2 \\
		&\quad \le -2\sum_{i=1,2}\sum_{x\in\Lambda}
		\int_0^t \psi^{(i)}_s(x)^{-}
		\Delta\psi^{(i)}_s(x)\,ds \\
		&\qquad {}-
		2\sum_{x\in\Lambda}\int_0^t
		\psi^{(1)}_s(x)^{-}
		\left(
		\epsilon_1^{-1}\left(\phi_s^{(1),\epsilon}(x)\right)^{-}
		-{\epsilon'_1}^{-1}\left(\phi_s^{(1),\epsilon'}(x)\right)^{-}
		\right)\,ds \\
		&\qquad {}-2\sum_{x\in\Lambda}\int_0^t
		\left(\psi^{(2)}_s(x)^{-}-\psi^{(1)}_s(x)^{-}\right) \\
		&\qquad\qquad\qquad 
		\times\left(\epsilon_2^{-1}
		\left(\phi_s^{(2),\Lambda,\epsilon,\delta}(x)
		-\phi_s^{(1),\Lambda,\epsilon,\delta}(x)\right)^{-}
		-\epsilon_2^{-1}
		\left(\phi_s^{(2),\epsilon'}(x)
		-\phi_s^{(1),\epsilon'}(x)\right)\right)\,ds \\
		&=:F_{1}+F_{2}+F_{3}.
	\end{align*}}%
	The same observation as in the proof of Lemma~\ref{semigr-lip}
	shows that $F_1$ and $F_3$ are nonpositive. 
	Here, since we have assumed $\epsilon_1\ge \epsilon_1'$, we obtain
	\begin{align*}
		F_2&\le -
		2\sum_{x\in\Lambda}\int_0^t
		\psi^{(1)}_s(x)^{-}
		\left(
		\epsilon_1^{-1}\left(\phi_s^{(1),\epsilon}(x)\right)^{-}
		-\epsilon_1^{-1}\left(\phi_s^{(1),\epsilon'}(x)\right)^{-}
		\right)\,ds.
	\end{align*}
	Using the same observation as in the proof of Lemma~\ref{semigr-lip}
	again, we obtain that $F_2$ is also nonpositive.
	Summarizing the above, we finally get 
	\[\sum_{i=1,2}\sum_{x\in\Lambda}
	(\psi^{(i)}_t(x)^{-})^2=0\]
	for every $t\ge0$, which indicates the conclusion.
\end{proof}

Lemma~\ref{lem5.2a} guarantees that
there exists a limit $\Phi_t^{(0,\epsilon_2)}$ of $\Phi_t^\epsilon$
as $\epsilon_1\downarrow0$.
Though we want to take the limit as $\epsilon_2\downarrow0$ also,
we can not expect the monotonicity in $\epsilon_2$ as in
Lemma~\ref{lem5.2a}. In order to establish the monotonicity
in some sense, let us transform the solution $\Phi_t^\epsilon$
as follows: we define $\Psi_t^\epsilon=(\psi_t^{(1),\epsilon},\psi_t^{(2),\epsilon})\in\real^\Lambda\times \real^\Lambda$ by
\begin{align*}
	\psi_t^{(1),\epsilon}(x)&=\frac{1}{\sqrt2}\left(
	\phi_t^{(1),\epsilon}(x)+\phi_t^{(2),\epsilon}(x)\right) \\
	\psi_t^{(2),\epsilon}(x)&=\frac{1}{\sqrt2}\left(
	-\phi_t^{(1),\epsilon}(x)+\phi_t^{(2),\epsilon}(x)\right).
\end{align*}
for every $x\in\Lambda$ and $t\ge0$, that is,
$\Psi_t^\epsilon(x)=(\psi_t^{(1),\epsilon}(x),\psi_t^{(2),\epsilon}(x))$ 
is the rotation of $\Phi_t^\epsilon(x)=(\phi_t^{(1),\epsilon}(x),\phi_t^{(1),\epsilon}(x))$ clockwise through $45$ degrees.
One can easily see that $\Psi_t^\epsilon$ solves SDEs
\begin{equation}\label{SDE2c}
\begin{cases}
	d\psi_t^{(1),\epsilon}(x)
	=\Delta\psi_t^{(1),\epsilon}dt+\sqrt{2}d\hat{w}_t^{(1)}
	+d\hat\ell_t^{(1),\epsilon}(x),& x\in\Lambda, \\
	d\psi_t^{(2),\epsilon}(x)
	=\Delta\psi_t^{(2),\epsilon}dt+\sqrt{2}d\hat{w}_t^{(2)}
	+d\hat\ell_t^{(2),\epsilon}(x)-d\hat\ell_t^{(1),\epsilon}(x), &x\in\Lambda, \\
	\psi_t^{(i),\epsilon}(x)\equiv\psi_0^{(i),\epsilon}(x)
	, &x\in\integer^d\smallsetminus\Lambda, i=1,2, \\
\end{cases}
\end{equation}
where $\hat\ell_t^{(1),\epsilon}(x)$
and $\hat\ell_t^{(2),\epsilon}(x)$
are defined by
\begin{align}
	\hat\ell_t^{(1),\epsilon}(x)&=
	\frac{1}{2}\epsilon_1^{-1}\int_0^t
	\left(\psi_s^{(1),\epsilon}(x)-\psi_s^{(2),\epsilon}(x)\right)^{-}\,ds, \label{eq4.2} \\
	\hat\ell_t^{(2),\epsilon}(x)&=
	2\epsilon_2^{-1}\int_0^t
	\left(\psi_s^{(2),\epsilon}(x)\right)^{-}\,ds. \label{eq4.3}
\end{align}
Here, $\hat{w}_t(x)=(\hat{w}_t^{(1)}(x),\hat{w}_t^{(2)}(x))$ is
also the rotation of $w_t(x)=(w_t^{(1)}(x),\phi_t^{(2)}(x))$ 
clockwise through $45$ degrees, and thus a two-dimensional standard
Brownian motion again. 
We can show the monotonicity for $\Psi_t^\epsilon$ in $\epsilon_2$
quite parallel to Lemma~\ref{lem5.2a}. 
\begin{lemma}\label{lem5.4a}
	Fix $\epsilon=(\epsilon_1,\epsilon_2),
	\epsilon'=(\epsilon'_1,\epsilon'_2)$ such that
	$\epsilon_1=\epsilon'_1$ and $\epsilon_2\ge\epsilon'_2$.
	Let $\Psi_t^{\Lambda,\epsilon}$ and $\Psi_t^{\Lambda,\epsilon'}$
	be solutions of \eqref{SDE2c} with common Brownian motions and
	initial datum.
	We then have $\Psi_t^{\Lambda,\epsilon}\le\Psi_t^{\Lambda,\epsilon'}$
	for every $t>0$.
\end{lemma}

Here, Lemmas~\ref{lem5.2a} and \ref{lem5.4a} imply the existence
of the limit in the right hand side of \eqref{eq5.1b}
for each $t\ge0$ and $x\in\Lambda$. Once we have
the boundedness and the equicontinuity of solutions $\Phi_t^\epsilon$
uniformly in $\epsilon$, we can conclude \eqref{eq5.1b}
by the similar way to the proof of Theorem~2.1 of \cite{DN03}.
We also obtain \eqref{eq2.3} as a simple application of 
the monotone convergence theorem.
As the final step for the proof of Theorem~\ref{existence-finite},
let us establish bounds corresponding to them.
\begin{lemma}\label{lem5.3a}
	Let $\Phi_t^\epsilon$ be the solution of \eqref{eq2.1}
	with initial data $\Phi_0^\epsilon$. Then, for every $T>0$,
	there exists a constant $C>0$ independent of $\epsilon_1,\epsilon_2>0$ such that
	\begin{gather}
		E\left[\sup_{0\le t\le T}\|\Phi_t^\epsilon\|^4\right]
		\le C(1+\|\Phi_0^\epsilon\|^4), \label{eq5.5a}\\
		E\left[\sup_{t_1\le t\le t_2}
		\|\Phi_t^\epsilon-\Phi_{t_1}^\epsilon\|^4\right]
		\le C(1+\|\Phi_0^\epsilon\|^4)\left((t_2-t_1)^2+(t_2-t_1)^3\right), \label{eq5.6a}
	\end{gather}
	holds for $0\le t_1\le t_2\le T$, where $\|\Phi\|$ denotes
	the ordinary Euclidean norm, that is,
	\[\|\Phi\|^2=\sum_{i=1,2}\sum_{x\in\Lambda}
	|\phi^{(i)}(x)|^2
	\]
	for $\Phi=(\phi^{(1)},\phi^{(2)})\in\real^\Lambda\times\real^\Lambda$.
\end{lemma}
\begin{proof}
For a fixed $\Psi=(\psi^{(1)},\psi^{(2)})\in\real^\Lambda\times\real^\Lambda$, we have
\begin{align*}
	\|\Phi_t^\epsilon-\Psi\|^2 
	&=	\|\Phi_0^\epsilon-\Psi\|^2 \\
	&\qquad {}+2\sum_{i=1,2}\sum_{x\in\Lambda}\int_0^t\left(\phi_s^{(i),\epsilon}(x)
	-\psi^{(i)}(x)\right)\Delta\phi_s^{(i),\epsilon}(x)\,ds \\
	&\qquad {}+2\sum_{x\in\Lambda}
	\int_0^t\left(\phi_s^{(1),\epsilon}(x)
	-\psi^{(1)}(x)\right)d\ell_s^{(1),\epsilon}(x) \\
	&\qquad {}+2\sum_{x\in\Lambda}
	\int_0^t\left(\phi_s^{(2),\epsilon}(x)-\phi_s^{(1),\epsilon}(x)
	-\psi^{(2)}(x)+\psi^{(1)}(x)\right)
	d\ell_s^{(2),\epsilon}(x) \\
	&\qquad {}+2t|\Lambda|+m_{t}(\Psi),
\end{align*}
where $m_{t}(\Psi)$ is the martingale with the following expression:
\[m_t(\Psi)=2\sum_{i=1,2}\sum_{x\in\Lambda}\int_0^t\left(\phi_s^{(i),\epsilon}(x)-\psi^{(i)}(x)\right)dw_s^{(i)}(x).\]
Here, since the third term in the right hand side is expressed by
\[\sum_{x\in\Lambda}\int_0^t\phi_s^{(1),\epsilon}(x)
d\ell_s^{(1),\epsilon}(x)
-\sum_{x\in\Lambda}\psi^{(1)}(x)
\ell_t^{(1),\epsilon}(x),\]
it is nonpositive by the definition of $\ell_s^{(1),\epsilon}(x)$. Similarly, we also obtain that the fourth term
is nonpositive. Using the Schwarz inequality, we get
\begin{equation}
	\|\Phi_t^\epsilon-\Psi\|^2
	\le \|\Phi_0^\epsilon-\Psi\|^2
	+K_1\int_0^t\left(1+\|\Phi_s^\epsilon\|^2+\|\Phi_s^\epsilon-\Psi\|^2
	\right)\,ds+m_t(\Psi)
\end{equation}
with some constant $K_1>0$ independent of $\epsilon$ and $t>0$.
Repeating the argument in the proof of Lemma~2.2 of \cite{DN03},
we obtain \eqref{eq5.5a} and \eqref{eq5.6a}.
\end{proof}

\subsection{The dynamics on the infinite lattice $\integer^d$}
\label{sec-existence-infinite}
We shall construct the infinite system $\Phi_t$ governed by SDEs \eqref{SDE1}.
Our goal is the following:
\begin{theorem}\label{existence-infinite}
	For every $\Phi_0=(\phi_0^{(1)},\phi_0^{(2)})\in\mathscr{X}_{r,+}^2$,
	there exists an unique (strong) solution $\Phi_t$ of \eqref{SDE1}.
	Furthermore, the identity \eqref{eq2.3d} holds.
\end{theorem}

Let us introduce two lemmas, which imply Theorem~\ref{existence-infinite}
by applying the same argument as in the proof of Theorem~2.1 of \cite{DN03}.
We note that $\Phi_t$ can be obtained as the increasing limit
of $\Phi_t^\Lambda$, and therefore the monotone convergence theorem yields
the identity \eqref{eq2.3d}.
\begin{lemma}\label{comparison5}
Let $\Phi_t^{\Lambda}=(\phi^{(1)}_t,\phi^{(2)}_t)$ and
$\Phi_t^{\Lambda'}=(\phi^{(1),\Lambda'}_t,\phi^{(2),\Lambda'}_t)$
be solutions of \eqref{SDE2} with common Brownian motions
$\{w_t^{(i)}(x);\,x\in\integer^d,i=1,2\}$ and initial datum satisfying
$\supp \Phi_0\subset \Lambda$. If $\Lambda\subset\Lambda'$, then
$\Phi_t^\Lambda\le\Phi_t^{\Lambda'}$ holds.
\end{lemma}
\begin{lemma}\label{lem5.3}
	Let $\Phi_t^\Lambda$ be the solution of \eqref{SDE2}
	with initial data $\Phi_0^\Lambda$. Then, for every $T>0$,
	there exists a constant $C>0$ independent of $\Lambda$ such that
	\begin{gather}
		E\left[\sup_{0\le t\le T}\|\Phi_t^\Lambda\|_r^4\right]
		\le C(1+\|\Phi_0^\Lambda\|_r^4), \label{eq5.5}\\
		E\left[\sup_{t_1\le t\le t_2}
		\|\Phi_t^\Lambda-\Phi_{t_1}^\Lambda\|_r^4\right]
		\le C(1+\|\Phi_0^\Lambda\|_r^4)(t_2-t_1)^2, \label{eq5.6}
	\end{gather}
	holds for $0\le t_1\le t_2\le T$.
	The estimates \eqref{eq5.5} and \eqref{eq5.6}
	is also true for the solution $\Phi_t$ of \eqref{SDE1}
	if there exists.
\end{lemma}

Let us give a proof of Lemma~\ref{lem5.3} only, since we can show
Lemma~\ref{comparison5} by a similar manner to the proof
of Proposition~\ref{comparison3}.
\begin{proof}[Proof of Lemma~\ref{lem5.3}]
For fixed $\Psi=(\psi^{(1)},\psi^{(2)})\in\mathscr{X}_{r,+}^2$, we have
\begin{align*}
	\|\Phi_t^\Lambda&-\Psi\|_{r}^2 \\
	&=	\|\Phi_0^\Lambda-\Psi\|_{r}^2 \\
	&\qquad {}+2\sum_{i=1,2}\sum_{x\in\Lambda}\exp(-2r|x|)\int_0^t\left(\phi_s^{(i),\Lambda}(x)
	-\psi^{(i)}(x)\right)\Delta\phi_s^{(i),\Lambda}(x)\,ds \\
	&\qquad {}+2\sum_{x\in\Lambda}\exp(-2r|x|)
	\int_0^t\left(\phi_s^{(1),\Lambda}(x)
	-\psi^{(1)}(x)\right)d\ell_s^{(1),\Lambda}(x) \\
	&\qquad {}+2\sum_{x\in\Lambda}\exp(-2r|x|)
	\int_0^t\left(\phi_s^{(2),\Lambda}(x)-\phi_s^{(1),\Lambda}(x)
	-\psi^{(2)}(x)+\psi^{(1)}(x)\right)
	d\ell_s^{(2),\Lambda}(x) \\
	&\qquad {}+2t\sum_{x\in\Lambda}\exp(-2r|x|)+m_{t}(\Psi),
\end{align*}
where $m_{t}(\Psi)$ is the martingale with the following expression:
\[m_t(\Psi)=2\sum_{i=1,2}\sum_{x\in\Lambda}\exp(-2r|x|)\int_0^t\left(\phi_s^{(i),\Lambda}(x)-\psi^{(i)}(x)\right)dw_s^{(i)}(x).\]
Here, since the third term in the right hand side is expressed by
\[\sum_{x\in\Lambda}\exp(-2r|x|)\int_0^t\phi_s^{(1),\Lambda}(x)
d\ell_s^{(1),\Lambda}(x)
-\sum_{x\in\Lambda}\exp(-2r|x|)\psi^{(1)}(x)
\ell_t^{(1),\Lambda}(x)\]
and $\ell_t^{(1),\Lambda}(x)$ is the local time of $\phi_t^{(1),\Lambda}(x)$,
this term is nonpositive. Similarly, we also obtain that the fourth term
is nonpositive. Using the Schwarz inequality and the Lipschitz continuity
of $\Phi=(\phi^{(1)},\phi^{(2)})\to(\Delta\phi^{(1)},\Delta\phi^{(2)})$ in $\mathscr{X}_r^2$, we get
\begin{equation}
	\|\Phi_t^\Lambda-\Psi\|_{r}^2
	\le \|\Phi_0^\Lambda-\Psi\|_{r}^2
	+K_1\int_0^t\left(1+\|\Psi_s^\Lambda\|_r^2+\|\Psi_s^\Lambda-\Psi\|_r^2
	\right)\,ds+m_t(\Psi)
\end{equation}
with some constant $K_1>0$ independent of $\Lambda$ and $t>0$.
Repeating the argument in the proof of Lemma~2.2 of \cite{DN03},
we obtain \eqref{eq5.5} and \eqref{eq5.6}.
We can also obtain the bound for $\Phi_t$ corresponding to 
\eqref{eq5.5} and \eqref{eq5.6} by the similar way to the above. 
\end{proof}

\section{Proof of the lower bound}\label{sec-lower}
We shall discuss the lower bound in Theorem~\ref{main-thm}.
Our goal in this section is the following:
\begin{theorem}\label{lower-bound}
	Under the condition of Theorem~\ref{main-thm}, we have
	\begin{gather}
		\liminf_{t\to\infty}E\left[\frac{\phi_t^{(i)}(0)}{\sqrt{\log_d(t)}}
		\right]\ge\sqrt{C_1/2}, \label{eq2.5b}\\
		\lim_{t\to\infty}P(\phi_t^{(i)}(0)\le \sqrt{(C_i-\epsilon)\log_d(t)})=0
		\label{eq2.6b}
	\end{gather}
	for every $i=1,2$ and $\epsilon>0$, where constants $C_1,C_2$ are same as in Theorem~\ref{main-thm}.
\end{theorem}
We at first prove the lower bound for the expectation \eqref{eq2.5b}
in Subsection~\ref{lower-exp}. As a next step, we establish an estimate for
the variance of $\phi_t^{(i)}(0)$ and conclude \eqref{eq2.6b}
in Subsection~\ref{estimate-var}.

\subsection{Estimate for the expected value}\label{lower-exp}
In this subsection, we establish the lower bound of the expected value
\eqref{eq2.5b}.
We at first prepare the following proposition, which 
plays the key role in the proof of \eqref{eq2.5b}.
It is based on the logarithmic Sobolev inequality for
$\mathscr{E}_{\epsilon,\delta}$.
\begin{prop}\label{LSI}
	Let $\mu_t^N$ be the law of $\phi_t^N$ with initial data
	$\phi_t^N\equiv\phi_t$ on $\Lambda_N$ and $\phi_t^N\equiv0$ on
	$\integer^d\smallsetminus\Lambda_N$. Then, there exist two constants
	$c_1,c_2>0$ such that
	\[d_{\mathrm{var}}(\mu_t^N,\mu_N)\le c_1N^d\exp(-c_2N^{-2}t).\]
	Here, $d_{\mathrm{var}}$ is the variational distance, that is,
	\[d_{\mathrm{var}}(\mu,\nu):=\sup_{A\in\mathscr{B}}|\mu(A)-\nu(A)|,\]
	where $\mathscr{B}$ is Borel $\sigma$-algebra in $(\real^{\Lambda_N})^2$.
\end{prop}
\begin{proof}
	Verifying the Bakry-Emery criteria (see \cite{DS89}), we obtain the logarithmic Sobolev
	inequality for $\mathscr{E}_{\epsilon,\delta}$:
	\[H(\mu_t^{N,\epsilon,\delta}|\mu^{N,\epsilon,\delta})
	\le\exp(-cN^{-2}t)H(\mu_0^{N,\epsilon,\delta}|\mu^{N,\epsilon,\delta}),\]
	where $\mu_0^{N,\epsilon,\delta}$ is the law of
	$\phi_t^{N,\epsilon,\delta}$ which is solution of \eqref{eq2.2}.
	Here, we denote the relative entropy by $H$, that is,
	\[H(\mu|\nu)=\begin{cases}
		\displaystyle
		E_{\nu}\left[\frac{d\mu}{d\nu}\log\frac{d\mu}{d\nu}\right],&
		\text{if }\mu<\nu, \\
		\infty,& \text{otherwise}.
	\end{cases}\]
	We therefore get
	\begin{align*}
		d_{\mathrm{var}}(\mu_t^N,\mu_N)&\le
		d_{\mathrm{var}}(\mu_t^N,\mu_t^{N,\epsilon})
		+d_{\mathrm{var}}(\mu_t^{N,\epsilon},\mu_t^{N,\epsilon,\delta})
		+d_{\mathrm{var}}(\mu_t^{N,\epsilon,\delta},\mu^{N,\epsilon,\delta}) \\
		&\qquad {}+d_{\mathrm{var}}(\mu^{N,\epsilon,\delta},\mu^{N,\epsilon})
		+d_{\mathrm{var}}(\mu^{N,\epsilon},\mu^{N}) \\
		&\le \exp(-cN^{-2}t)H(\mu_0^{N,\epsilon,\delta}|\mu^{N,\epsilon,\delta})
		+d_{\mathrm{var}}(\mu_t^N,\mu_t^{N,\epsilon})
		+d_{\mathrm{var}}(\mu_t^{N,\epsilon},\mu_t^{N,\epsilon,\delta}) \\
		&\qquad {}+d_{\mathrm{var}}(\mu^{N,\epsilon,\delta},\mu^{N,\epsilon})
		+d_{\mathrm{var}}(\mu^{N,\epsilon},\mu^{N}).
	\end{align*}
	Here, using \eqref{eq2.3} and \eqref{eq2.4}, all terms except the first
	converge to $0$ as $\delta\downarrow0, \epsilon_1\downarrow0$
	and $\epsilon_2\downarrow0$. Therefore, it suffices
	to show
	\begin{equation}\label{eq2.3f}
		H(\mu_0^{N,\epsilon,\delta}|\mu^{N,\epsilon,\delta})\le CN^d
	\end{equation}
	with a constant $C>0$ independent of $N,\epsilon,\delta$.
	Calculating the relative entropy, we obtain
	\begin{align*}
	H(\mu_0^{N,\epsilon,\delta}|\mu^{\epsilon,\delta})
	& = 2\log Z_{\Lambda_N,0} + \log Z_{\Lambda_N,\epsilon,\delta} \\
	&\qquad {}+\sum_{x\in\Lambda_N}E_{\rho}\left[\log f(\phi^{(1)}(x),
	\phi^{(2)}(x)) 
	+W_{\epsilon,\delta}(\phi^{(1)}(x),\phi^{(2)}(x))\right] \\
	&\qquad {}+\sum_{i=1,2}E_{\rho}\left[H_{\Lambda}(\phi^{(i)})\right]
	\end{align*}
	Since $W_{\epsilon,\delta}(u)\equiv0$ on the support of $\rho$,
	we obtain $E_\rho[W_{\epsilon,\delta}(\phi(x))]=0$
	for every $x\in\Lambda_N$.
	Noting that $W_{\epsilon,\delta}$ is non-negative, we obtain
	$\log Z_{\Lambda_N,\epsilon,\delta}\le0$.
	Since $N^{-d}\log Z_{\Lambda_N,0}$ converges to $\sigma(0)$ when $N\to\infty$,
	where $\sigma$ is the (unnormalized) surface tension
	(see \cite{DGI00} or \cite{FS97}),
	this term is bounded above by $CN^d$ for some constant $C>0$.
	Noting that 
	\[\sum_{x\in\Lambda_N}E_\rho[\log f(\phi^{(1)}(x),\phi^{(2)})(x)]+
	\sum_{i=1,2}E_\rho\left[H_\Lambda(\phi^{(i)})\right]
	\le CN^d\]
	holds for some constant $C>0$, we obtain the desired estimate \eqref{eq2.3f}.
\end{proof}

Since we have finished the preparation, we are at the position to show
the lower bound for the expected value.
\begin{proof}[Proof of \eqref{eq2.5b}]
We at first note that
\begin{align*}
	\lim_{N\to\infty}&\mu_\Lambda\left(|\Lambda_{\kappa N}|^{-1}
	\sum_{x\in \Lambda_{\kappa N}}\phi^{(i),N}(x)\le \sqrt{(C_i-\eta)\log_{d}(N)}
	\right)=0
\end{align*}
holds for every $i=1,2,\,\eta>0$ and $0<\kappa<1$,
see \cite{BG03} (when $d\ge3$) or \cite{Sa03} (when $d=2$).
Choosing $N=t^{1/2-\epsilon}$ and combining the above
with Proposition~\ref{LSI}, we obtain
\begin{align}
	\lim_{t\to\infty}&P\left(|\Lambda_{\kappa N}|^{-1}
	\sum_{x\in \Lambda_{\kappa N}}\phi_t^{(i),N}(x)\le \sqrt{(C_i-\eta)\log_{d}(N)}
	\right)=0, \label{eq2.5}
\end{align}
for every $i=1,2,\,\eta>0$ and $0<\kappa<1$.
Since we get
\begin{align*}
	E\left[\phi_t^{(i)}(0)\right]
	&=E\left[|\Lambda_{\kappa N}|^{-1}
	\sum_{x\in \Lambda_{\kappa N}}\phi_t^{(i)}(x)\right] \\
	&\ge E\left[|\Lambda_{\kappa N}|^{-1}
	\sum_{x\in \Lambda_{\kappa N}}\phi_t^{(i),N}(x)\right] \\
	&\ge \sqrt{(C_i-\eta)\log_{d}(N)}\,
	P\left(
	|\Lambda_{\kappa N}|^{-1}
	\sum_{x\in \Lambda_{\kappa N}}\phi_t^{(i),N}(x)> \sqrt{(C_i-\eta)\log_{d}(N)}	\right).
\end{align*}
for every $t>0$ and $N\ge1$ from
the invariance of the initial data and coefficients of \eqref{SDE1}
under the spatial shift, Proposition~\ref{comparison3} and the
Chebyshev inequality, we obtain
\begin{equation}\label{eq2.7}
	\liminf_{t\to\infty}E\left[\frac{\phi_t^{(i)}(0)}{\sqrt{\log_d(t)}}
	\right]\ge \sqrt{(C_i-\eta)/2}.
\end{equation}
from \eqref{eq2.5}. Since $\eta>0$ is arbitrary, we obtain the conclusion.
\end{proof}

\subsection{Estimate for the asymptotic variance}\label{estimate-var}
Since we have already shown the estimate \eqref{eq2.5b} for the expected value, it suffices to show that the variance
of $\phi_t^{(1)}(0)$ is small enough in order to obtain \eqref{eq2.6b}.
Our goal in this subsection is the following:
\begin{theorem}\label{asymp-var}
	There exist two constants $K_1,K_2>0$ independent of $t$ such that
	\[\mathrm{var}(\phi_t^{(i)}(0))\le K_1\int_0^{K_2t}p_s(0,0)\,ds\]
	holds for every $t>0$ and $i=1,2$.
\end{theorem}
Once we obtain Theorem~\ref{asymp-var}, we can immediately conclude
\eqref{eq2.6b} from \eqref{eq2.5b}, since
\[\lim_{t\to\infty}\frac{K_1}{\log_d(t)}\int_0^{K_2t}p_s(0,0)\,ds=0\]
holds for every $d\ge2$. 

\begin{proof}[Proof of Theorem~\ref{asymp-var}]
	In \cite{DZ03}, the random walk representation for the variance
	of the dynamics with reflection is established.
	Once we have the random walk representation
	for our system similarly to \cite{DZ03},
	Theorem~\ref{asymp-var} can immediately be obtained as its corollary.
	However, we shall show Theorem~\ref{asymp-var} using the approximation
	from the dynamics with weak reflection introduced by
	\eqref{eq2.1} and \eqref{eq2.2}.
	Applying \eqref{eq2.4}, \eqref{eq2.3}, \eqref{eq2.3d},
	we only have to show the following: there exist constants
	$K_1,K_2>0$ independent of $\Lambda,\epsilon,\delta$ and $t$ such that
	\begin{equation}
		\mathrm{var}(\phi_t^{(i),\Lambda,\epsilon,\delta}(0))\le K_1
		\int_0^{K_2t}p_s(0,0)\,ds
	\end{equation}
	holds for every $t>0$ and $i=1,2$.

	For simple notation, during the proof we denote 
	$\phi_t^{(i),\Lambda,\epsilon,\delta}$ and
	$\Phi_t^{\Lambda,\epsilon,\delta}$
	simply by 
	$\phi_t^{(i)}$ and $\Phi_t$, respectively.
	Let $F^{(i)}=\phi^{(i)}(0)$. For fixed $t>0$,
	we introduce a function $u^{(i)}(s,\Phi)$ by
	\[u^{(i)}(s,\Phi)=(P_{t-s}F^{(i)})(\Phi),\]
	where $P_{t}$ is the semigroup associated with the Markov process
	$\Phi_t$. Then, it is easy to see that
	$M_s^{(i)}:=u^{(i)}(s,\Phi_s),\,0\le s\le t$ is continuous martingale
	written as
	\begin{equation}\label{eq2.12b}
	M_s^{(i)}=\sqrt{2}\sum_{x\in\Lambda}\sum_{j=1,2}\int_0^s
	\frac{\partial u^{(i)}}{\partial \phi^{(j)}(x)}(r,\phi_r)\,dw_r^{(j)}.
	\end{equation}
	Note that one can easily see that $u^{(i)}(s,\cdot)$ is a function
	of $C^\infty$-class, since $W_{\epsilon,\delta}$ is $C^\infty$-class.
	Here, from the definition of $M^{i}$, we have
	\[\mathrm{var}(\phi_t^{(i)})=E\left[\langle M^{(i)}\rangle_t\right].\]
	Using the expression \eqref{eq2.12b}, we get
	\[E\left[\langle M^{(i)}\rangle_t\right]
	=2\sum_{x\in\Lambda}\sum_{j=1,2}\int_0^t
	E\left[\left(\frac{\partial u^{(i)}}{\partial \phi^{(j)}(x)}(r,\phi_r)\right)^2\right]\,dr.\]	Here, since Proposition~\ref{semigr-lip} implies
	\[\sup_{\phi\in(\real^{\Lambda})^2}
	\left|\frac{\partial u^{(i)}}{\partial \phi^{(j)}(x)}(r,\phi)\right|
	\le p_{t-r}(0,x),\]
	we obtain
	\[E\left[\langle M^{(i)}\rangle_t\right]
	\le 4\sum_{x\in\Lambda}\int_0^t p_{t-r}(0,x)\,dr
	\le 4\int_0^t p_{2r}(0,0)\,dr,\]
	which shows the conclusion.
\end{proof}

\section{Proof of the upper bound}\label{sec-upper}
In this section, we shall establish the upper bound of the height.
Our goal is the following:
\begin{theorem}\label{upper-bound}
	Under the condition of Theorem~\ref{main-thm}, we have
	\begin{gather}
		\limsup_{t\to\infty}E\left[\frac{\phi_t^{(i)}(0)}{\sqrt{\log_d(t)}}
		\right]\le\sqrt{C_i/2}, \label{eq3.1b}\\
		\lim_{t\to\infty}P\left(
		\phi_t^{(1)}(0)\ge \sqrt{(C_i/2+\epsilon)\log_d(t)}
		\right)=0, \label{eq3.3b}
	\end{gather}
	for every $i=1,2,\,\epsilon>0$, where constants $C_1,C_2$ are same as in Theorem~\ref{main-thm}.
\end{theorem}

We shall show Theorem~\ref{upper-bound} by comparing
with suitable dynamics.
We at first introduce another dynamics of interfaces:
the upper interface is on the lower and the lower interface
is not pushed down by the upper. Let $\tilde\Phi_t=(\tilde\phi_t^{(1)},
\tilde\phi_t^{(2)})$ be the solution of the following equation
\begin{equation}\label{SDE2a}
\begin{cases}
	d\tilde\phi_t^{(1)}(x)=\Delta\tilde\phi_t^{(1)}(x)dt+\sqrt{2}dw_t^{(1)}(x)
	+d\tilde\ell_t^{(1)}(x),& x\in\integer^d, \\
	d\tilde\phi_t^{(2)}(x)=\Delta\tilde\phi_t^{(2)}(x)dt+\sqrt{2}dw_t^{(2)}(x)
	+d\tilde\ell_t^{(2)}(x), & x\in\integer^d,
\end{cases}
\end{equation}
where $\tilde\ell_t^{(1)}(x)$ and $\tilde\ell_t^{(2)}(x)$ are the local
times of $\tilde\phi_t^{(1)}(x)$ and $\tilde\phi_t^{(2)}(x)-\tilde\phi_t^{(1)}(x)$ at $0$, respectively. The such process can be constructed by the following way:
we at first construct the process $\tilde\phi_t^{(1)}$, and after that, we construct the process $\tilde\phi_t^{(2)}$ over $\tilde\phi_t^{(1)}$.
The following proposition implies that the system $\tilde\Phi_t$
always stays above the original system $\Phi_t$ and therefore we only need to
establish the upper bound for \eqref{SDE2a}.
\begin{prop}\label{prop3.1}
	Let $\Phi_t$ and $\tilde\Phi_t$ be solutions of
	\eqref{SDE1} and \eqref{SDE2a} with common Brownian motions
	and initial data, respectively. We then
	have $\Phi_t\le \tilde\Phi_t$ for every $t>0$.
\end{prop}
\begin{proof}
	Let $\tilde\Phi_t^\Lambda=(\tilde\phi_t^{(1),\Lambda},
	\tilde\phi_t^{(2),\Lambda})$ be the strong solution of the following
	SDEs:
	\begin{equation}\label{SDE2a-fin}
	\begin{cases}
		d\tilde\phi_t^{(1)}(x)=\Delta\tilde\phi_t^{(1)}(x)dt+\sqrt{2}dw_t^{(1)}(x)
		+d\tilde\ell_t^{(1)}(x),& x\in\Lambda, \\
		d\tilde\phi_t^{(2)}(x)=\Delta\tilde\phi_t^{(2)}(x)dt+\sqrt{2}dw_t^{(2)}(x)
		+d\tilde\ell_t^{(2)}(x), & x\in\Lambda, \\
		\tilde\phi_t^{(1)}(x)=\tilde\phi_t^{(2)}(x)=0,& x\in \Lambda^\complement.
	\end{cases}
	\end{equation}
	It is sufficient for the goal to show $\Phi_t^\Lambda\le\tilde\Phi_t^\Lambda$
	for every finite $\Lambda\subset\integer^d$,
	since $\Phi_t$ and $\tilde\Phi_t$ can be obtained as
	increasing limits of $\Phi_t^\Lambda$ and $\tilde\Phi_t^\Lambda$
	as $\Lambda\uparrow\integer^d$, respectively.
	We define $\psi_t^{(i)}$ by $\psi_t^{(i)}=\phi_t^{(i),\Lambda}
	-\tilde\phi_t^{(i),\Lambda}$.
	Let us at first show that $\phi_t^{(1),\Lambda}\le
	\tilde\phi_t^{(1),\Lambda}$.
	Calculating $(\psi_t^{(1)}(x)^+)^2$, we have
	\begin{align*}
		\left(\psi_t^{(1)}(y)^+\right)^2&=
		2\int_0^t\psi_s^{(1)}(y)^+\Delta\psi_s^{(1)}(y)\,ds \\
		&\quad {}+2\int_0^t\psi_s^{(1)}(y)^+
		\left(d\ell_s^{(1),\Lambda}(x)-d\tilde\ell_s^{(1),\Lambda}(x)\right) \\
		&\quad {}-2\int_0^t\psi_s^{(1)}(y)^+
		d\ell_s^{(2),\Lambda}(x) \\
		&=:F_1+F_2+F_3.
	\end{align*}
	By the definition of $\ell_t^{(1),\Lambda}(x),\,\ell_s^{(2),\Lambda}(x)$
	and $\tilde\ell_t^{(1),\Lambda}(x)$,
	the terms $F_2$ and $F_3$ are nonpositive.
	By the same calculation as in the first step of the proof of
	Proposition~\ref{semigr-lip}, we obtain that 
	\[\sum_{x\in\Lambda}\left(\psi_t^{(1)}(x)^+\right)^2=0,\]
	which shows $\phi_t^{(1),\Lambda}\le\tilde\phi_t^{(1),\Lambda}$.
	
	Next, we shall show $\phi_t^{(2),\Lambda}\le\tilde\phi_t^{(2),\Lambda}$.
	Calculating $(\psi_t^{(1)}(x)^+)^2$, we obtain
	\begin{align*}
		\left(\psi_t^{(2)}(x)^+\right)^2&=
		2\int_0^t\psi_s^{(2)}(x)^+
		\Delta\psi_s^{(2)}(x)\,ds \\
		&\quad {}+2\int_0^t
		\psi_s^{(2)}(x)^+
		\left(d\ell_s^{(2),\Lambda}(x)-d\tilde\ell_s^{(2),\Lambda}(x)\right).
	\end{align*}
	Noting that $\phi_s^{(1),\Lambda}\le\tilde\phi_s^{(1),\Lambda}$,
	we have
	\[\phi_s^{(2),\Lambda}(x)-\phi_s^{(1),\Lambda}(x)
	>\tilde\phi_s^{(2),\Lambda}(x)-\tilde\phi_s^{(2),\Lambda}(x)\]
	if $\psi_s^{(2)}(x)^+>0$ holds.
	We therefore obtain that the second term is nonpositive.
	Similarly to above, we conclude
	$\psi_t^{(2)}=\phi_t^{(2),\Lambda}-\tilde\phi_t^{(2),\Lambda}\le0$
	for every $t>0$.
\end{proof}

Theorem~1.1
of \cite{DN03} indicates the upper bound \eqref{eq3.1b} and
\eqref{eq3.3b} for $i=1$. We shall focus our attention to the upper bound
for $\tilde\phi_t^{(2)}$. However, since we also have the lower bound
for $\tilde\phi_t^{(1)}$ (see \cite{DN03}), we only need to show
the followings bounds for $\tilde\psi_t:=\tilde\phi_t^{(2)}-
\tilde\phi_t^{(1)}$:
\begin{gather}
	\limsup_{t\to\infty}E\left[\frac{\tilde\psi_t(0)}{\sqrt{\log_d(t)}}
	\right]\le\sqrt{C_2/2}-\sqrt{C_1/2}, \label{eq3.5b}\\
	\lim_{t\to\infty}P\left(
	\tilde\psi_t(0)\ge (\sqrt{C_2/2}-\sqrt{C_1/2}+\epsilon)\sqrt{\log_d(t)}
	\right)=0, \label{eq3.6b}
\end{gather}
for every $\epsilon>0$.

We furthermore introduce the stochastic dynamics
$\hat\Phi_t=(\hat\phi_t^{(1)},\hat\phi_t^{(2)})$ by the following SDEs:
\begin{equation}\label{SDE2b}
\begin{cases}
	d\hat\phi_t^{(1)}(x)=\Delta\hat\phi_t^{(1)}(x)dt+\sqrt{2}dw_t^{(1)}(x),& x\in\integer^d, \\
	d\hat\phi_t^{(2)}(x)=\Delta\hat\phi_t^{(2)}(x)dt+\sqrt{2}dw_t^{(2)}(x)
	+d\hat\ell_t^{(2)}(x), & x\in\integer^d,
\end{cases}
\end{equation}
where $\hat\ell_t^{(2)}(x)$ are the local
time of $\hat\phi_t^{(2)}(x)-\hat\phi_t^{(1)}(x)$ at $0$.
This dynamics can be regarded as the dynamics of interfaces such that
one interface is on the other which is dynamics without wall.
We shall then compare $\tilde\psi_t$ with
$\hat\psi_t:=\hat\phi_t^{(2)}-\hat\phi_t^{(1)}$.
Repeating the argument in the proof of Proposition~\ref{prop3.1},
we obtain the following proposition:
\begin{prop}\label{prop3.2}
	Let $\tilde\Phi$ and $\hat\Phi$ be solutions of
	\eqref{SDE2a} and \eqref{SDE2b} with common Brownian motions
	and initial data, respectively.
	We then have $\tilde\psi_t\le \hat\psi_t$
	for every $t>0$.
\end{prop}

Now, we are at the position to show \eqref{eq3.5b} and \eqref{eq3.6b}.
By Proposition~\ref{prop3.2}, it is sufficient to show the similar bound
for the $\hat\psi_t$. We can easily show that $\hat\psi_t$ has the same
law as $\sqrt{2}\rho_t$, where $\rho_t$ is the solution of the following
equation:
\begin{equation*}
	d\rho_t=\Delta\rho_t(x)\,dt+\sqrt{2}d\bar{w}_t(x)+d\bar\ell_t(x),
	\quad x\in\integer^d.
\end{equation*}
Here, $\bar{w}_t=\{\bar{w}_t(x);x\in\integer^d\}$ is a family of
independent one-dimensional Brownian motions and $\bar\ell_t(x)$
is the local time of $\rho_t(x)$ at $0$. Applying 
the upper bound for $\hat\rho_t$:
\begin{gather*}
\lim_{t\to\infty}P(\rho_t(0)\ge \sqrt{C'\log_d(t)})=0 \\
\lim_{t\to\infty}E\left[\frac{\rho_t(0)}{\sqrt{\log_d(t)}}\right]=\sqrt{C_1/2}
\end{gather*}
for every $C'>C_1$ (see \cite{DN03}), we conclude \eqref{eq3.5b} and \eqref{eq3.6b}
and therefore the upper bound \eqref{eq3.1b} and \eqref{eq3.3b} for $i=2$.

\section*{Acknowledgments}
The author would like to thank to 
Professor J.-D.~Deuschel of TU Berlin, Professor T.~Funaki of University of
Tokyo, Professor H.~Sakagawa of Keio University, who gave me helpful advice.

\bibliographystyle{amsplain}
\providecommand{\bysame}{\leavevmode\hbox to3em{\hrulefill}\thinspace}
\providecommand{\MR}{\relax\ifhmode\unskip\space\fi MR }
\providecommand{\MRhref}[2]{%
  \href{http://www.ams.org/mathscinet-getitem?mr=#1}{#2}
}
\providecommand{\href}[2]{#2}

\end{document}